\documentclass{article}
\usepackage[nosumlimits]{amsmath}
\usepackage{graphicx}
\usepackage{amssymb,amsthm}
\def\MM#1{\boldsymbol{#1}}
\newcommand{\pp}[2]{\frac{\partial #1}{\partial #2}}

\def\MM#1{\boldsymbol{#1}}
\DeclareMathOperator{\diff}{d}

\DeclareMathOperator{\CG}{CG}
\DeclareMathOperator{\DG}{DG}
\DeclareMathOperator{\BDM}{BDM}
\DeclareMathOperator{\RT}{RT}
\DeclareMathOperator{\BDFM}{BDFM}
\DeclareMathOperator{\B}{B}
\usepackage{amscd}
\usepackage{natbib}
\newtheorem{theorem}{Theorem}
\newtheorem{definition}[theorem]{Definition}
\newtheorem{proposition}[theorem]{Proposition}
 
\usepackage[top=2.5cm, bottom=2.5cm, left=3.5cm, right=3.5cm]{geometry}
\usepackage{fancybox}
\begin{document}
\title{Compatible finite element methods for numerical weather prediction}
\author{C. J. Cotter and A. T. T. McRae}
\maketitle

\begin{abstract}
  This article takes the form of a tutorial on the use of a particular
  class of mixed finite element methods, which can be thought of as
  the finite element extension of the C-grid staggered finite
  difference method. The class is often referred to as compatible
  finite elements, mimetic finite elements, discrete differential
  forms or finite element exterior calculus. We provide an elementary
  introduction in the case of the one-dimensional wave equation,
  before summarising recent results in applications to the rotating
  shallow water equations on the sphere, before taking an outlook
  towards applications in three-dimensional compressible dynamical
  cores.
\end{abstract}

\section{Introduction}
Mixed finite element methods are a generalisation of staggered finite
difference methods, and are intended to address the same problem:
spurious pressure mode(s) observed in the finite difference A-grid and
are also observed in finite element methods when the same finite
finite element method, different finite element spaces are selected
for different variables. The vast range of available finite element
spaces is both a blessing and a curse, and many different combinations
have been proposed, analysed and used for large scale geophysical
fluid dynamics applications, particularly in the ocean modelling
community
\cite{Ro+2005,LeRoPo2007,LePo2008,Da2010,CoHaPa2009,CoHa2011,rostand2008raviart,le2012spurious,comblen2010practical},
whilst many other combinations have been used in engineering
applications where different scales and modelling aspects are
important.

In this article we limit the scope considerably by discussing only one
particular family of mixed finite element methods, known variously as
compatible finite elements (the term that we shall use here), mimetic
finite elements, discrete differential forms or finite element
exterior calculus. These finite element methods have the important
property that differential operators such as grad and curl map from
one finite element space; these embedding properties lead to discrete
versions of the div-curl and curl-grad identities of vector
calculus. This echoes the important ``mimetic'' properties of the
C-grid finite difference method, discussed in full generality on
unstructured grids, and shown to be highly relevant and useful in
geophysical fluid dynamics applications, in
\cite{ThRiSkKl2009,RiThKlSk2010}. This generalisation of the C-grid
method is very useful since it allows (i) use arbitrary grids, with no
requirement of orthogonal grids, without loss of
consistency/convergence rate; (ii) extra flexibility in the choice of
discretisation to optimize the ratio between global velocity degrees
of freedom (DoFs) and global pressure DoFs to eliminate spurious mode
branches; and (iii) the option to increase the consistency/convergence
order.

Compatible finite element methods were first identified in the 1970s,
and quickly became very popular amongst numerical analysts since this
additional mathematical structure facilitated proofs of stability and
convergence, and provided powerful insight. These results were
collected and unified in \cite{BrFo1991}, an excellent book which has
been out of print for a long time, but a new addition has recently
appeared \cite{boffi2013mixed}. These methods have become the standard
tool for groundwater modelling using Darcy's law
\cite{allen1985mixed}, and have also become very popular for solving
Maxwell's equations \cite{bossavit1988whitney,hiptmair2002finite}
where a mathematical structure based on differential forms was
developed by Bossavit, together with the term ``discrete differential
forms''. This structure was enriched, extended and unified under the
term ``finite element exterior calculus'' by Douglas Arnold and
collaborators \cite{ArFaWi2006,ArFaRa2010}, who used the framework to
develop new stable discretisations for elasticity. This has produced a
rich and beautiful theory in the language of exterior calculus,
however, in this article we shall only use standard vector calculus
notation.

In geophysical applications, the stability properties of compatible
finite elements have long been recognised, leading to various choices
being proposed and analysed on triangular meshes
\cite{WaCa1998,rostand2008raviart}. However, no explicit use was made
of the compatible structure beyond stability until
\cite{cotter2012mixed}, which proved that all compatible finite
element methods have exactly steady geostrophic modes; this is
considered a crucial property for numerical weather prediction
\cite{StTh2012}. They also showed that the global DoF balance between
velocity and pressure in compatible finite element discretisations is
crucial in determining the existence of spurious mode branches. If
there are less than 2 velocity DoFs of freedom per pressure DoF, then
there will be branches of spurious inertia-gravity wave branches. This
occurs in the RT0-DG0 (lowest order Raviart-Thomas space for velocity,
and piecewise constant for pressure) on triangles, as demonstrated by
Danilov in \cite{Da2010}. In fact, if the Rossby radius is not
resolved\footnote{The barotropic Rossby radius is rarely resolved in
  ocean models, and the higher baroclinic Rossby radii (there are
  $n-1$ of these in an $n$-layer model) get smaller and smaller. This
  means that there are unresolved Rossby radii in all numerical
  weather prediction models.}, then the spurious and physical branches
merge leading to disasterous results. On the other hand, if the
velocity/pressure global DoF ratio is greater than 2:1, there will be
spurious Rossby modes similar to those observed for the hexagonal
C-grid \cite{Th08}. \cite{cotter2012mixed} identified the BDFM1-DG1
spaces (1st Brezzi-Douglas-Fortin-Marini space for velocity, linear
discontinuous space for pressure) on triangles, and the RTk-DG1 ($k$th
Raviart-Thomas space for velocity, $k$th order discontinuous space for
pressure) on quadrilaterals as having potential for numerical weather
prediction, due to their exact global 2:1 DoF ratio.

Recently, there has been some further developments in the application
of compatible finite element methods to the nonlinear rotating shallow
water equations on the sphere, as part of the UK Dynamical Core ``Gung
Ho'' project. An efficient software implementation of the compatible
finite element spaces on the sphere was provided in \cite{Ro+2013},
whilst an energy-enstrophy conserving formulation analogous to the
energy-enstrophy conserving C-grid formulation of \cite{ArLa1981} was
provided in \cite{McRaCo2014}. A discussion of these properties in the
language of finite element exterior calculus was provided in
\cite{CoTh2014}.

In this article we provide an elementary introduction to compatible
finite element methods, for a reader equipped with knowledge of vector
calculus, and calculus of variations. We begin with a simple linear
one-dimensional example in Section \ref{1d}. We then discuss
applications to the linear and nonlinear shallow-water equations in
Section \ref{sw}. Finally, we provide a glimpse of current work
developing three dimensional formulations for numerical weather
prediction in Section \ref{outlook}.

\section{One dimensional formulation}
\label{1d}
In this section we develop the compatible finite element method in the
context of the one-dimensional scalar wave equation on the domain
$[0,L]$ with periodic boundary conditions,
\begin{equation}
\label{eq:wave eqn}
h_{tt} - h_{xx} = 0, \quad h(0,t)=h(L,t).
\end{equation}
It is more relevant to issues arising in the shallow water equations,
and beyond, to split this equation into two first order equations,
in the form
\begin{equation}
\label{eq:dual wave eqn}
u_t + h_x = 0, \quad h_t + u_x = 0, \quad 
h(0,t)=h(L,t),\,u(0,t)=u(L,t).
\end{equation}
In this section, we shall discretise Equation \eqref{eq:dual wave
  eqn} in space using compatible finite element methods.

In general, the finite element method is based on two key ideas: (i)
the approximation of the numerical solution by functions from some
chosen finite element spaces, and (ii) the weak form. We shall
motivate the latter by discussing the former in the context
of Equation \eqref{eq:dual wave eqn}.
\begin{definition}[Finite element space (on a one-dimensional domain)]
  We partition the interval $[0,L]$ into $N_e$ non-overlapping
  subintervals, which we call \emph{elements}; the partition is called
  a \emph{mesh}. We shall call the point shared by two neightbouring
  elements a \emph{vertex}. A finite element space is a collection of
  functions on $[0,L]$ which are:
  \begin{enumerate}
  \item polynomials of some specified maximum degree $p$
    when restricted to each element $e$, and
  \item have some specified degree of continuity (discontinuous,
    continuous, continuous derivative, \emph{etc}.)
  \end{enumerate}
\end{definition}
The most common options for continuity are continuous functions, in
which case we name the finite element space $\CG(p)$ for given $p$, and
discontinuous functions, in which case we name the finite element
space $\DG(p)$ (higher order continuity finite element spaces are more
exotic, B-splines for example, and we shall not discuss them here).
An example function from the $\CG1$ space and an example function from
the $\DG0$ space are shown in Figure \ref{p11d}. We use the term finite
element space since the collection of functions form a vector space
(\emph{i.e.}, they may be added together and scaled by real numbers,
and addition and scaling satisfy the required properties of a vector
space). This makes finite element spaces amenable to the tools of
linear algebra. We also note that finite element spaces are finite
dimensional. This makes them amenable to calculation on a computer.

\begin{figure}
\includegraphics[width=8cm]{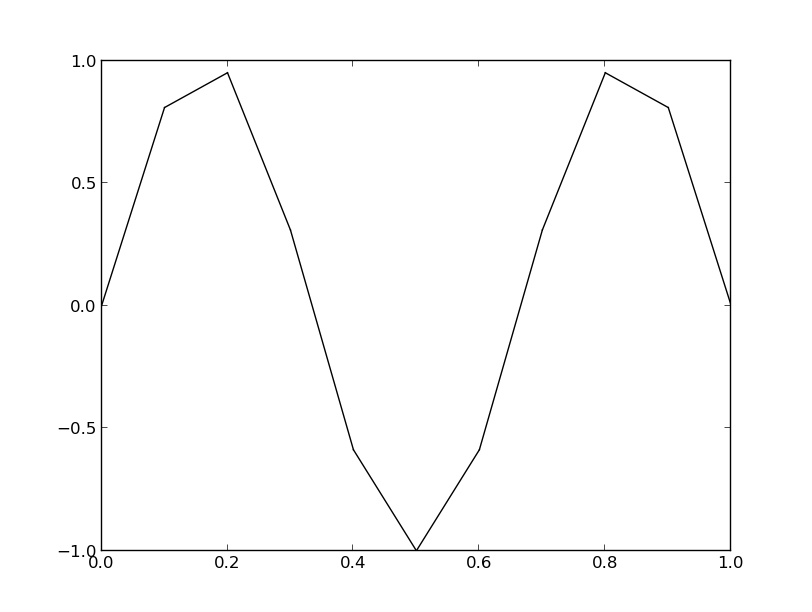}
\includegraphics[width=8cm]{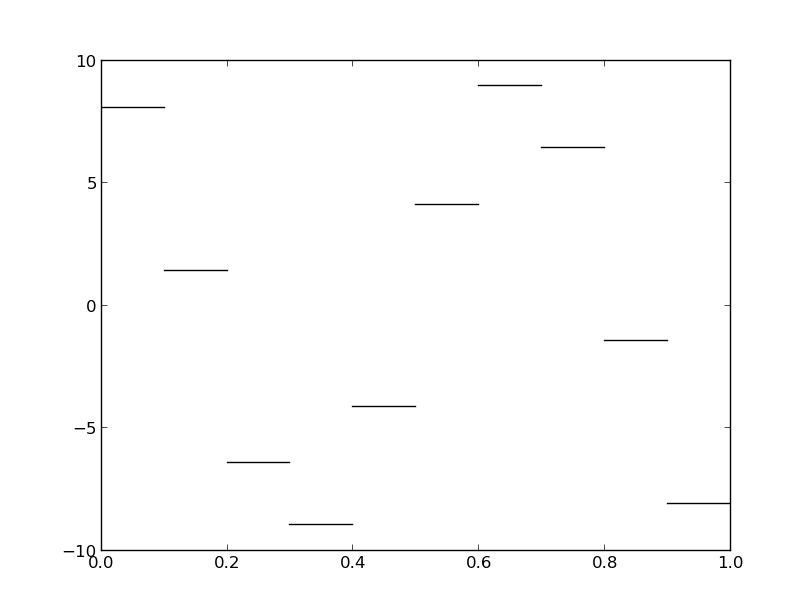}
\caption{\label{p11d}Example finite element functions for a
  subdivision of the domain $[0,1]$ into 10 elements.  Left: A
  function from the $\CG1$ finite element space.  Right: A function from
  the $\DG0$ finite element space.}
\end{figure}

We now proceed to discretise Equation \eqref{eq:dual wave eqn}. We
would like to restrict both $h$ and $u$ to finite element spaces, let
us say $\CG1$ for the purposes of this discussion (we will use the
notation $u\in CG1$ to mean that $u$ is a function in the finite
element space $\CG1$). Clearly we do not obtain solutions of
\eqref{eq:dual wave eqn}, since if $u\in\CG1$ then $u_x\in \DG0$. (In
fact, we obtained the $\DG0$ function in Figure \ref{p11d} by taking
the derivative of the $\CG1$ function.) Hence, we choose to find the
best possible approximation to \eqref{eq:dual wave eqn} by minimising
the magnitude of $u_t+h_x$ and $h_t+u_x$ whilst keeping $u$ and $h$ in
$\CG1$. To do this, we need to choose a way of measuring the magnitude
of functions (\emph{i.e.}, a norm), of which the $L^2$ norm, given by
\begin{equation}
\|u\|_{L^2} = \sqrt{\int_0^L u^2 \diff x},
\end{equation}
is the most natural and computationally feasible. Our finite
element approximation becomes\footnote{The factors of $1/2$ are
not significant but they simplify the subsequent equations.}
\begin{equation}
\label{eq:optimisation}
\min_{u_t\in \CG1}\frac{1}{2}\|u_t+h_x\|^2, \quad
\min_{h_t\in \CG1}\frac{1}{2}\|h_t+u_x\|^2.
\end{equation}
The standard calculus of variations approach to finding the minimiser
$u_t$ for the first of Equations \eqref{eq:optimisation} follows from
noting if $u_t$ is optimal, then infinitesimal changes in $u_t$ do not
change the value of $\|u_t+h_x\|^2$. This is expressed mathematically
as
\begin{equation}
\lim_{\epsilon \to 0} \frac{ \frac{1}{2}\|u_t+\epsilon w +h_x\|^2
- \frac{1}{2}\|u_t + h_x\|^2}{\epsilon},
\end{equation}
for any $w\in \CG1$ (we adopt the notation $\forall w \in \CG1$). We
obtain
\begin{align}\nonumber
  0 & = \lim_{\epsilon\to0}
\frac{1}{2}\|u_t+\epsilon w + h_x\|^2 - \frac{1}{2}\|u_t+h_x\|^2, \\
\nonumber  & = \frac{1}{2}\int_0^L (u_t+\epsilon w + h_x)^2\diff x 
  - \frac{1}{2}\int_0^L (u_t+h_x)^2 \diff x, \\
\nonumber & = \frac{1}{2}\int_0^L (u_t + h_x)^2 + 2\epsilon w(u_t+h_x)\diff x 
  - \frac{1}{2}\int_0^L (u_t+h_x)^2 \diff x, \\
 & = \int_0^L w(u_t+h_x)\diff x, \quad \forall w \in \CG1.
\label{eq:up1p1}
\end{align}
An identical calculation follows for the $h_t$ equation, and we 
obtain
\begin{equation}
\label{eq:pp1p1}
0 = \int_0^L \phi(u_t+h_x)\diff x, \quad \forall \phi \in \CG1.
\end{equation}
We refer to $w$ and $\phi$ as \emph{test functions}. Note that
Equations (\ref{eq:up1p1}-\ref{eq:pp1p1}) can be directly obtained by
multiplying Equations \eqref{eq:dual wave eqn} by test functions $w$
and $\phi$, and integrating over the domain; the minimisation process
is only a theoretical tool to explain that it is the best possible
approximation using the finite element space. Since these equations
represent the error-minimising approximations of Equations
\eqref{eq:dual wave eqn} with $u\in \CG1$, $p\in \CG1$, we call them
the \emph{projections} of Equations \eqref{eq:dual wave eqn} onto
$\CG1$. In general, the projection of an function or an equation onto
a finite element space is called \emph{Galerkin projection}.

Equations (\ref{eq:up1p1}-\ref{eq:pp1p1}) can be implemented
efficiently on a computer by expanding $w$, $\phi$, $u$ and $h$ in a
basis over $\CG1$,
\begin{equation}
u(x) = \sum_{i=1}^n N_i(x)u_i, \quad h(x) = \sum_{i=1}^n N_i(x)h_i, \quad
w(x) = \sum_{i=1}^n N_i(x)w_i, \quad \phi(x) = \sum_{i=1}^n N_i(x)\phi_i,
\end{equation}
with real valued basis coefficients $u_i$, $w_i$, $h_i$,
$\phi_i$. These basis coefficients are still functions of time since
we have not discretised in time yet. In general, in one dimension it
is always possible to find a basis for $\CG(p)$ and $\DG(p)$ finite
element spaces such that the basis functions $N_i$ are non-zero in at
most two (neighbouring) elements. For details of the construction of
basis functions for $\CG$ and $\DG$ spaces of arbitrary $p$, see
\cite{KaSh2005}. Substitution of these basis expansions into Equations
(\ref{eq:up1p1}-\ref{eq:pp1p1}) gives
\begin{equation}
\MM{\rm w}^T\left(M\dot{\MM{\rm u}}+D{\MM{\rm h}}\right)=0, \quad
\MM{\phi}^T\left(M\dot{\MM{\rm h}}+D{\MM{\rm u}}\right)=0,
\label{eq:bilinear form}
\end{equation}
where $M$ and $D$ are matrices with entries given by
\begin{equation}
M_{ij} = \int_0^L N_i(x)N_j(x)\diff x, \quad
D_{ij} = \int_0^L N_i(x)\pp{N_j}{x}(x)\diff x,
\end{equation}
and $\MM{\rm u}$, $\MM{\rm h}$, $\MM{\rm w}$ and $\MM{\phi}$ are
vectors of basis coefficients with $\MM{\rm u}=(u_1,u_2,\ldots,u_n)$
\emph{etc.} Equations \eqref{eq:bilinear form} must hold for all test
functions $w$ and $\phi$, and therefore for arbitrary coefficient
vectors $\MM{\rm w}$ and $\MM{\phi}$. Therefore, we obtain the
matrix-vector systems,
\begin{equation}
\label{eq:matvec}
M\dot{\MM{\rm u}}+D{\MM{\rm h}}=0, \quad
M\dot{\MM{\rm h}}+D{\MM{\rm u}}=0.
\end{equation}
Having chosen a basis where each basis function vanishes in all but
two elements, the matrices $M$ and $D$ are extremely sparse and hence
can be assembled efficiently. For details of the efficient assembly
process, see \cite{KaSh2005}. Furthermore, the matrix $M$ is
well-conditioned and hence can be cheaply inverted using iterative
methods \cite{wathen1987realistic}. It remains to integrate Equations
\eqref{eq:matvec} using a discretisation in time. A generalisation of
this approach is used for all of the finite element methods that we
describe in this paper.

One key problem with Equations (\ref{eq:up1p1}-\ref{eq:pp1p1}) is that
of spurious modes. For example, if we use a regular grid of $N_e$
elements of the same size, if $h$ is a $\CG1$ ``zigzag'' function that
alternates between $1$ and $-1$ between each vertex, then $h_x$ is a
$\DG0$ ``flip-flop'' function that takes the value $\Delta x$ and
$-\Delta x$ in alternate elements, where $\Delta x =
N_e/L$. Multiplication by a $\CG1$ test function $w$ and integrating
then gives zero for arbitrary $w$. The easiest way to understand why
is to choose $w$ to be a hat-shaped basis function that is equal to 1
at a single vertex, and 0 at all other vertices. Then the integral of
$w$ multiplied by $h_x$ is a (scaled) average of $h_x$ over two
elements, which is equal to zero. Since all $w$ can be expanded in
basis functions of this form, we obtain $D\MM{\rm w}$ in every
case. This is a problem because our original zigzag function is very
oscillatory, and so the approximation of the derivative should be
large. In general, using the same finite element space for $u$ and $h$
leads to the existence of spurious modes which have very small
numerical derivatives, despite being very oscillatory, and hence
propagate very slowly. When nonlinear terms are introduced, these
modes get coupled to the smooth part of the function, and grow
rapidly, making the numerical scheme unusable.

In finite difference methods, this problem is avoided by using
staggered grids, with different grid locations for $u$ and $h$. In
finite element methods, the analogous strategy is to choose different
finite element spaces for $u$ and $h$. This is referred to as a
\emph{mixed finite element method}. We shall write $u \in V_0$, $h \in
V_1$ and discuss different choices for $V_0$ and $V_1$. In particular,
we shall choose $V_0=\CG1$ and $V_1=\DG0$, together with the
higher-order extensions $V_0=\CG(p)$ and $V_1=\DG(p-1)$, for some
chosen $p>1$. The reason for doing this is that if $u\in V_0$, then
$u_x\in V_1$: this is because $u$ is continuous but can have jumps in
the derivative, and differentiation reduces the degree of a polynomial
by 1. We say that the finite element spaces $V_0$ and $V_1$ are
\emph{compatible} with the $x$-derivative. This choice means that
$h_t+u_x\in V_1$ and there is no approximation in writing that
equation. Put another way, the Galerkin projection in Equation
\ref{eq:pp1p1} is ``trivial'', \emph{i.e.} it does not change the
equation.

To write down our compatible finite element method we have one further
issue to address, namely that $h\in V_1$ is discontinuous, and so
$h_x$ is not globally defined. This is dealt with by integrating the
$h_x$ term by parts in the finite element approximation, and we obtain
\begin{align*}
\int_0^L wu_t - w_xh \diff x &= 0, \quad \forall w\in V_0, \\
\int_0^L \phi(h_t + u_x)\diff x &= 0, \quad \forall \phi\in V_1.
\end{align*}
There is no boundary term arising from integration by parts due to the
periodic boundary conditions. Three out of the four terms in these two
equations involve trivial projections that do nothing, the $u_t$,
$h_t$ and $u_x$ terms. This means that they introduce no further
errors beyond approximating the initial conditions in the finite
element spaces. The only term that we have to worry about is the
discretised $h_x$ term, where we would like to convince ourselves that
there are no spurious modes. This is done by showing that the
following mathematical condition holds.
\begin{definition}[inf-sup condition]
The spaces $V_0$ and $V_1$ satisfy the inf-sup condition\footnote{Here
  sup is short for \emph{supremum}, which is the maximum value, roughly
  speaking.}  if there exists a constant $C>0$, independent of the
choice of mesh, such that
\begin{equation}
\sup_{w\in V_0, \, w\ne 0}\frac{\left|\int_0^L w_x h \diff x\right|}{\|w_x\|_{L^2}}
 \geq
C\|h\|_{L^2},
\end{equation}
for all non-constant $h\in V_1$.
\end{definition}
This prevents spurious modes because it says that for any non-constant
$h$, there exists at least one $w$ such that the integral is
reasonably large in magnitude compared to the size of $w_x$ and
$h$. In general, proving the inf-sup condition for mixed finite
element methods is a fairly technical business (see \cite{AuBrLo2004}
for a review). However, for our compatible finite element
discretisation it is completely straightforward.
\begin{proposition}[inf-sup condition for compatible finite elements]
Let $V_0$ and $V_1$ be chosen such that if $h\in V_1$ is non-constant
then we can find $w\in V_0$ such that $w_x=h$. Then the inf-sup
condition is satisfied, with $C=1$.
\end{proposition}
\begin{proof}
  For any non-constant $h$, take $w'$ such that $w'_x=h$ (which is
  possible by the assumption of the proposition). Then
\begin{equation}
\sup_{w\in V_0}\frac{\left|\int_0^L w_x h \diff x\right|}
{\|w_x\|_{L^2}}
 \geq
\frac{\left|\int_0^L w'_x h \diff x\right|}
{\|w'_x\|_{L^2}}
 =
\frac{\|h\|^2_{L^2}}{\|h\|_{L^2}} = \|h\|_{L^2}.
\end{equation}
\end{proof}
The assumption of the proposition is true whenever $V_0=\CG(p)$,
$V_1=\DG(p-1)$. This means that our compatible finite element methods
will be free from spurious modes. This concludes our discussion of the
one-dimensional case; in the next section we discuss the construction
of two-dimensional compatible finite element spaces and their
application to the shallow water equations.

\section{Application to linear and nonlinear shallow water equations
  on the sphere}
\label{sw}
In one dimension, a compatible finite element method was described for
the wave equation, by choosing two finite element spaces $V_0$ and
$V_1$, such that (i) if $u\in V_0$, then $u_x\in V_1$; and (ii) if
$p\in V_1$ then there is $u \in V_0$ such that $u_x=p$. We represent
this structure by the following diagram.
\begin{equation*}
  \begin{CD}
   \underbrace{{V}_0}_{\mbox{\tiny Continuous}}@>\partial_x>>
    \underbrace{{V}_1}_{\mbox{\tiny Discontinuous}}
  \end{CD}
\end{equation*}
Recall that functions in $V_0$ are continuous, which allows 
the derivative to be globally defined. 

In two dimensions, we take our computational domain, denoted $\Omega$,
to be a rectangle in the plane with periodic boundary
conditions\footnote{The construction here can be extended to the
  surface of a sphere, or in fact any two-dimensional manifold
  embedded into $\mathbb{R}^3$. For details see \cite{Ro+2013}.}, and
partition it into either triangular or quadrilateral elements. We then
obtain compatible finite element methods by choosing a sequence of
three finite element spaces, ($V_0$, $V_1$, $V_2$), such that:
\begin{enumerate}
\item $V_0$ contains scalar-valued, continuous functions, satisfying
  periodic boundary conditions.
\item $V_1$ contains vector-valued functions $\MM{u}$, satisfying
  periodic boundary conditions, with $\MM{u}\cdot\MM{n}$ continuous
  across element edges, where $\MM{n}$ is the normal to the boundary
  between two neighbouring elements. The tangential component need not
  be continuous. This is sufficient continuity for $\nabla\cdot\MM{u}$
  to be defined globally.
\item $V_2$ contains scalar-valued functions that may be discontinuous
  across element boundaries.
\item If $\psi\in V_0$ then $\nabla^\perp\psi=(-\psi_y,\psi_x)\in
  V_1$.
\item If $\MM{u}\in V_1$ with $\nabla\cdot\MM{u}=0$ and $\int_\Omega
  \MM{u}\diff x = 0$, then there exists some $\psi\in V_0$ with
  $\MM{u}=\nabla^\perp\psi$.
\item If $\MM{u}\in V_1$ then $\nabla\cdot\MM{u}\in V_2$.
\item If $D \in V_2$ with $\int_\Omega D\diff x=0$ then there exists
  some $\MM{u}\in V_1$ with $\nabla\cdot\MM{u}=D$.
\end{enumerate}
We represent this structure in the following diagram.
  \begin{equation*}
  \begin{CD}
    \underbrace{{V}_0}_{\mbox{\tiny Continuous}}@>\nabla^\perp >>
    \underbrace{{V}_1}_{\mbox{\tiny Continuous normal components}}
    @>\nabla\cdot>> \underbrace{{V}_2}_{\mbox{\tiny Discontinuous}}
  \end{CD}
\end{equation*}
One example of compatible finite element spaces is $V_0=\CG2$,
$V_1=\BDM1$, $V_2=\DG0$ on triangles. The finite element space $\BDM1$
contains vector-valued functions that have linear $x$- and
$y$-components in each element, with continuous normal components
across element boundaries. These finite element spaces are illustrated
in Figure \ref{fig:BDM1_DOFs}. If $\psi\in \CG1$, then $\psi$ is
continuous and the gradient can be uniformly evaluated. Since $\psi$
is continuous, the value of $\psi$ is the same along both sides of
each boundary between two elements. This means that the component of
$\nabla\psi$ tangential to the boundary is continuous. The component
of $\nabla\psi$ in the direction normal to the boundary may
jump. Transforming $\nabla\psi$ to $\nabla^\perp\psi$ rotates the
vector by 90 degrees, and so $\nabla^\perp\psi$ has continuous normal
component but possibly discontinuous tangential
component. Furthermore, if $\psi$ is a quadratic polynomial within
each element, then $\nabla\psi$ is linear. Hence, we conclude that
$\nabla^\perp\psi\in \BDM1$, confirming property 4 above. Similar
reasoning by this type of inspection confirms the remaining properties
5-7.
\begin{figure}
\centerline{\includegraphics[width=10cm]{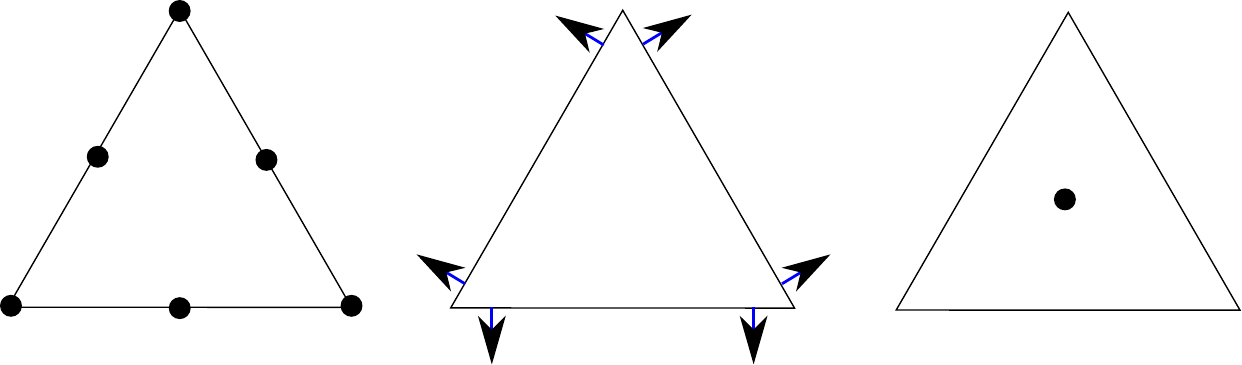}}
\vspace{3mm}
\centerline{\includegraphics[width=14cm]{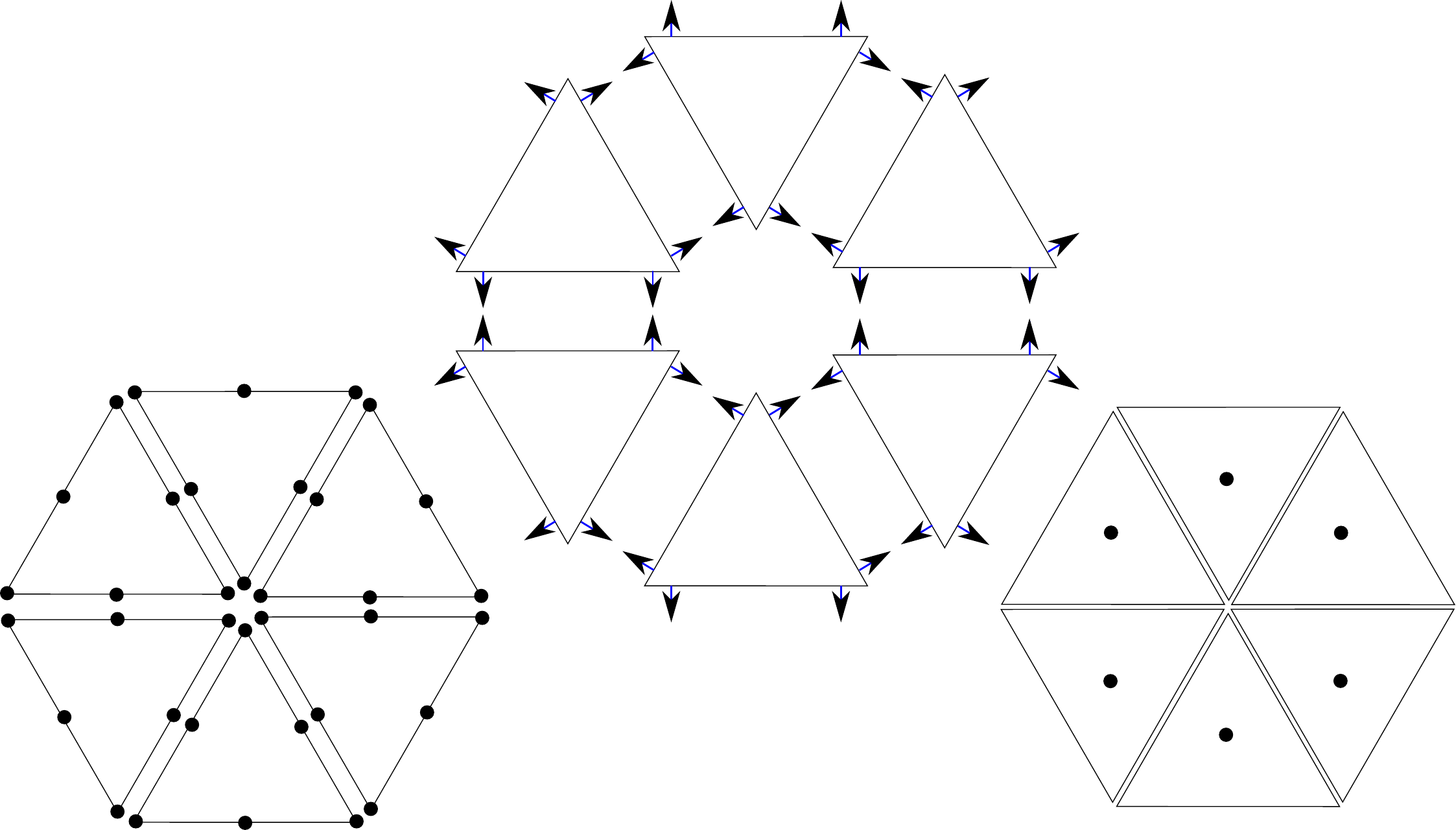}}
\caption{\label{fig:BDM1_DOFs}{\bfseries Top:} Diagrams illustrating
  $V_0=\CG2$ (left), $V_1=\BDM1$ (middle), $V_2=\DG0$ (right), on
  triangles.  These diagrams show the node points, depicted as
  circular dots, for a nodal basis for each finite element space on
  one triangle. Each basis function in the nodal basis is equal to 1
  on one node point, and 0 on all other node points, and vary
  continuously in between (since the basis functions are polynomials
  within one element). There is one basis function per node point, and
  the basis coefficient corresponding to that basis function is the
  value of the finite element function at that node.  In the case of
  $V_1$, the basis functions are vector-valued, and each node point
  has a unit vector associated with it, indicated by an arrow. Each
  basis function dotted with the unit vector is equal to 1 on one node
  point, and 0 at all the other nodes. The basis coefficient
  corresponding to that basis function is the value of the finite
  element function dotted with the unit vector at that
  node. {\bfseries Bottom:} Diagrams illustrating the continuity of
  functions between several neighbouring elements. Shared node points
  on adjoining elements mean that a single basis coefficient is used
  for these nodes. This leads to full continuity between elements for
  $\CG2$ functions, continuity in the normal components between
  elements for $\BDM1$ (since only the normal components of vectors
  are shared), and no implied continuity for $\DG0$.}
\end{figure}

There is a whole range of compatible finite element spaces that
satisfy properties 1-7. A particular choice generally depends on: (i)
what shape of elements we want, (ii) what relative size of the
dimensions of $V_0$, $V_1$, and $V_2$ we want, and (iii) what order of
accuracy we want (which depends on the degree of the polynomials
used). In \cite{cotter2012mixed}, it was shown that
$\dim(V_1)=2\dim(V_2)$ (in the case of periodic boundary conditions)
is a desirable property. If $\dim(V_1)<2\dim(V_2)$ then there are
spurious inertia-gravity waves which are known to cause serious
problems in simulating balanced flow \cite{Da2010}. If
$\dim(V_1)>2\dim(V_2)$ then there are spurious Rossby waves, the
nature of which is somewhat less clear. On triangles,
\cite{cotter2012mixed} identified the following choice as satisfying
$\dim(V_1)=2\dim(V_2)$: $V_0=\CG2+\B3$ (continuous quadratic functions
plus a cubic ``bubble'' function that vanishes on element boundaries),
$V_1=\BDFM1$ (quadratic vector-valued functions that are constrained
to have linear normal components along each triangle edge, with
continuous normal components), and $V_2=\DG1$. This choice is
illustrated in Figure \ref{BDFM1_DOFs}.
\begin{figure}
\centerline{\includegraphics[width=10cm]{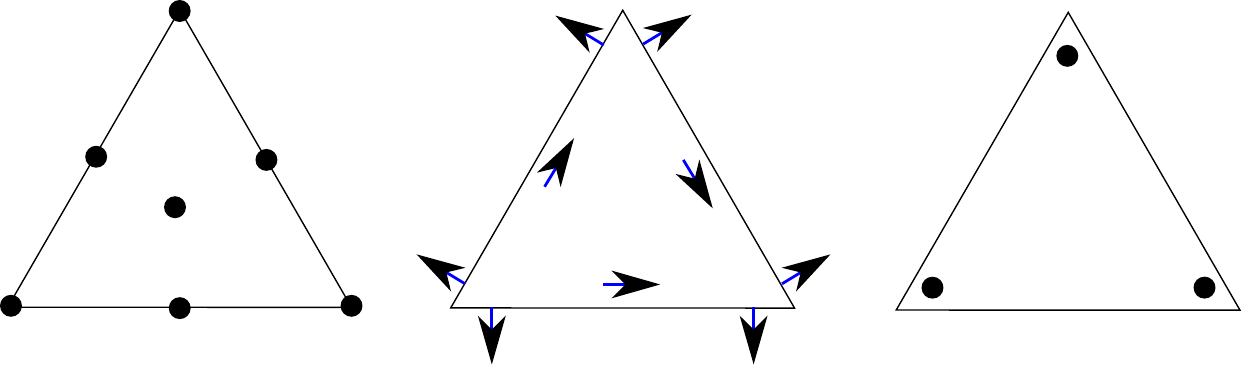}}
\caption{\label{BDFM1_DOFs}Diagrams illustrating $\CG2+\B3$ (left),
  $\BDFM1$ (middle), $\DG1$ (right), on triangles. $\CG2+\B3$ is an
  enrichment of $\CG2$, illustrated in Figure \ref{fig:BDM1_DOFs},
  adding the cubic function which vanishes on the boundary of the
  triangle. $\BDFM1$ is an enrichment of $\BDM1$ to include quadratic
  vector-valued functions that have vanishing normal components on the
  boundary of the triangle (there are three of these). The extra
  degrees of freedom are tangential components on the edge
  centres. Since tangential components are not required to be
  continuous, these values are not shared by neighbouring elements and
  there will be one tangential component node on each side of each
  edge. $\BDFM1$ has 9 nodes in each element, but 6 of these are
  shared with other elements, so $\dim(\BDFM1)=6N_e$ on the periodic
  plane. $\DG1$ has 3 nodes in each element, none of which are shared,
  so $\dim(\DG1)=3N_e$, and hence $\dim(\BDFM1)=2\dim(\DG1)$ as
  required. }
\end{figure}
On quadrilaterals, \cite{cotter2012mixed} identified the same property
in the choice $(V_0,V_1,V_2)=(\CG1,\RT0,\DG0)$, and the higher-order
extensions $(\CG(p),\RT(p-1),\DG(p-1))$. The first two sets of spaces
in this family are defined and illustrated in Figure \ref{RTQDOFS}.
\begin{figure}
\centerline{\includegraphics[width=10cm]{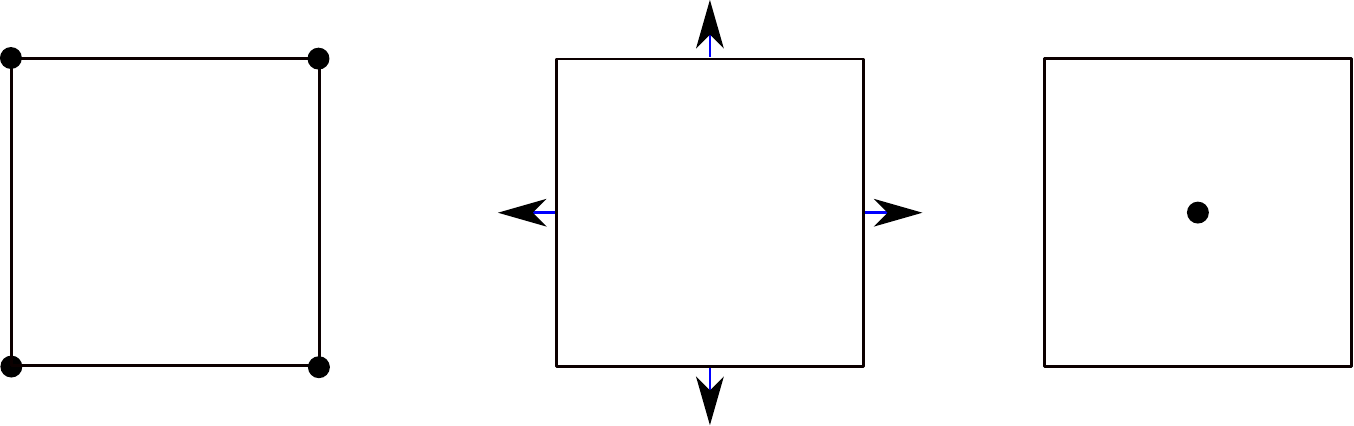}}
\centerline{\includegraphics[width=10cm]{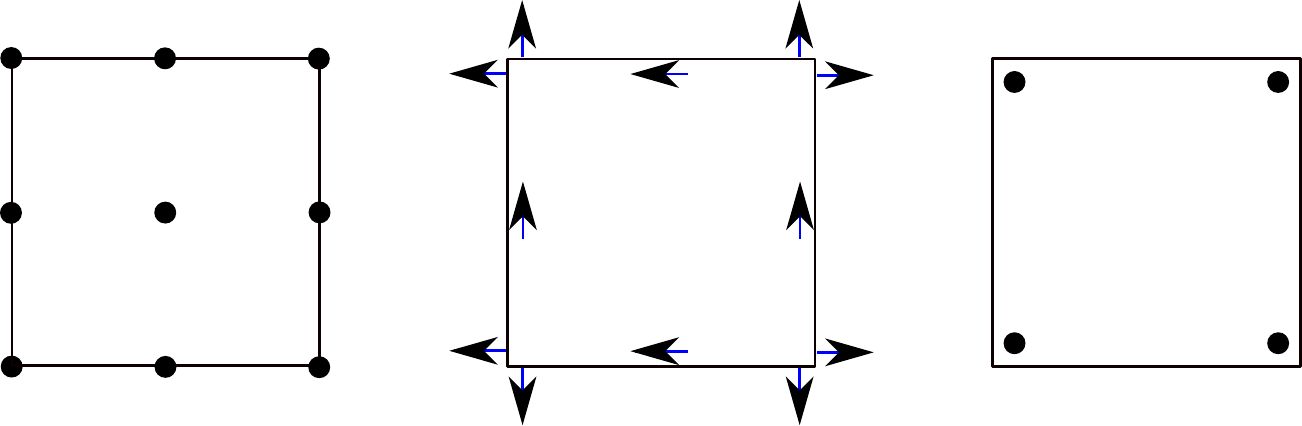}}
\caption{\label{RTQDOFS}{\bfseries Top}: Diagram illustrating
  $V_0=\CG1$, $V_1=\RT0$, $V_2=\DG0$, on quadrilaterals. On squares,
  $\CG1$ means bilinear, \emph{i.e.} linear function of $x$ multiplied
  by linear function of $y$. $\RT0$ functions are vector-valued, with
  the $x$-component being linear in $x$ and constant in $y$, and the
  $y$-component being constant in $y$ and linear in $x$. {\bfseries
    Bottom}: Diagram illustrating $V_0=\CG2$, $V_1=\RT1$, $V_2=\DG1$,
  where $\CG2$ are biquadratic functions (product of quadratic
  function of $x$ and quadratic function of $y$), $\RT1$ have
  $x$-component quadratic in $x$ and linear in $y$, and $y$-component
  linear in $x$ and quadratic in $y$.}
\end{figure}

In this discussion and the accompanying figures, we have defined the
finite element spaces on equilateral triangles or on squares. To
define them on nonsymmetric triangles/squares, including curved
elements to approximate the surface of a sphere, we apply geometric
transformations to the basis functions defined in the regular case.
In the case of $V_1$, special transformations (known as Piola
transformations) are required that preserve the normal components of
the vector-valued function on element boundaries and, in the case of
curved elements in three dimensions, keeps the vectors tangential to
the surface element. For details of these transformations, together
with their efficient implementation, see \cite{RoKiLo2009,Ro+2013}.

We now use these finite element spaces to the linear rotating shallow
water equations on the $f$-plane, explaining along the way why we
consider the compatible finite element method to be an extension of 
the C-grid finite difference method. The model equations are
\begin{align}
\label{u eqn}
\MM{u}_t + f\MM{u}^\perp + g\nabla h &=0, \\
\label{h eqn}
h_t + H\nabla\cdot\MM{u} &=0,
\end{align}
where $\MM{u}$ is the horizontal velocity, $h$ is the layer depth, $f$
is the (constant) Coriolis parameter, $g$ is the acceleration due to
gravity, $H$ is the (constant) mean layer depth, and
$\MM{u}^\perp=(-u_2,u_1)$.

For our finite element approximation, we choose $\MM{u}\in V_1$ and $h
\in V_2$. Extending the methodology of the previous section
(\emph{i.e.}  multiplying Equation \eqref{u eqn} by a test function
$\MM{w}\in V_1$, integrating the pressure gradient term by parts (the
boundary term vanishes due to the periodic boundary conditions),
multiplying Equation \eqref{h eqn} by a test function $\phi\in V_2$,
and integrating both equations over the domain $\Omega$), we obtain
\begin{align}
\label{discrete u eqn}
\int_\Omega\MM{w}\cdot\MM{u}_t\diff x + \int_\Omega f\MM{w}\cdot\MM{u}^\perp\diff x
 - \int_\Omega g\nabla\cdot\MM{w}\nabla h \diff x &=0, \quad \forall 
\MM{w}\in V_1, 
\\
\label{discrete h eqn}
\int_{\Omega} \phi\left(h_t + H\nabla\cdot\MM{u}\right)\diff x &=0,
\forall \phi \in V_2.
\end{align}
Since $h_t+H\nabla\cdot\MM{u}\in V_2$, the projection of Equation
\eqref{h eqn} is trivial, \emph{i.e.} Equation \eqref{h eqn} is
satisfied exactly under this discretisation. This means that we only
need to scrutinise the discretisation of the Coriolis and the pressure
gradient terms. In the previous section, we discussed the inf-sup condition
for one dimensional compatible finite element methods. The equivalent
condition in the two dimensional (and in fact, three dimensional) case
is 
\begin{equation}
\sup_{\MM{w}\in V_1}\frac{\left|\int_{\Omega} \nabla\cdot\MM{w} h \diff x\right|}
{\|\nabla\cdot\MM{w}\|_{L^2} }
\geq
C\|h\|_{L^2},
\end{equation}
for all non-constant $h\in V_2$. In the compatible finite element
case we again obtain $C=1$, and can provably avoid spurious pressure
modes.

Regarding the Coriolis term, the crucial condition for large scale
balanced flow (\emph{i.e.}, large scale numerical weather prediction)
is that the numerical discretisation supports exact geostrophic
balance. In the model equations (\ref{u eqn}-\ref{h eqn}), if
$\nabla\cdot\MM{u}=0$ and $\int_\Omega\MM{u}\diff x=0$, then \eqref{h
  eqn} implies that $h_t=0$. We have $\MM{u}=\nabla^\perp\psi$ for
some streamfunction $\psi$. If we choose $gh=f\psi$ then
\begin{equation}
\MM{u}_t = -f\MM{u}^\perp - g\nabla h = \nabla (f\psi - g h) =0,
\end{equation}
and we have a steady state, which we call \emph{geostrophic balance}.
If we allow $f$ to vary with $y$, leading to Rossby waves, or
introduce nonlinear terms, then this state of geostrophic balance
starts to evolve on a slow timescale relative to the rapidly
oscillating gravity waves. In large scale flow, the weather system
stays close to this balanced state. If we wish to predict the long
time evolution of this state accurately, it is essential that the
numerical discretisation exactly reproduces steady geostrophic states,
otherwise the errors in representing this balance will lead to
spurious motions that are much larger than the slow evolution when
Rossby waves or nonlinear evolution are introduced, and the forecast
will be useless. The C-grid staggering exactly reproduces steady
geostrophic states, if the Coriolis term is represented correctly (see
\cite{ThRiSkKl2009,ThCo2012} for how to do this on very general
grids); this accounts for the popularity and success of the C-grid in
numerical weather prediction. 

We now demonstrate that compatible finite element methods also have
steady geostrophic states under the same conditions (this was first
shown in \cite{cotter2012mixed}). If $\nabla\cdot\MM{u}=0$ and
$\int_\Omega\MM{u}\diff x = 0$ for $\MM{u}\in V_1$, then
$h_t=-H\nabla\cdot\MM{u}=0$. Next, we can find $\psi\in V_0$ with
$\MM{u}=\nabla^\perp\psi$. Then we solve for $h\in V_2$ from the
equation
\begin{equation}
\label{eq:project psi}
\int_\Omega \phi g h \diff x = \int_\Omega \phi f \psi \diff x, \quad
\forall \phi \in V_2,
\end{equation}
\emph{i.e.}, $gh$ is the projection of $f\psi$ into $V_2$. Then
\begin{align*}
\int_\Omega \MM{w}\cdot\MM{u}_t\diff x & = -\int_\Omega f\MM{w}\cdot\MM{u}^\perp
\diff x + \int_\Omega \nabla\cdot\MM{w}gh\diff x, \\
 & = \int_\Omega f\MM{w}\cdot\nabla\psi
\diff x + \int_\Omega \nabla\cdot\MM{w}gh\diff x, \\
 & = -\int_\Omega \nabla\cdot\MM{w}f\psi
\diff x + \int_\Omega \nabla\cdot\MM{w}gh\diff x, \\
& = 0,
\end{align*}
where we may integrate by parts in the third line since $\MM{w}$ has
continuous normal component and $\psi$ is continuous which means that
the integration by parts is exact, and where we note in the fourth
line that $\nabla\cdot\MM{w}\in V_2$, and so we may use Equation
\eqref{eq:project psi} with $\phi=\nabla\cdot\MM{w}$, meaning that we
obtain 0 in the final line. Hence, we have an exact steady state.

The extension to the nonlinear shallow water equations makes use 
of the \emph{vector invariant form},
\begin{align}
\label{u eqn nonlinear}
\MM{u}_t + qh\MM{u}^\perp + \nabla \left(gh + \frac{1}{2}|\MM{u}|^2\right) &=0,\\
h_t + \nabla\cdot(h\MM{u}) &= 0, 
\label{h eqn nonlinear}
\end{align}
where $q$ is the shallow water potential vorticity
\begin{equation}
\label{eq:q}
q= \frac{\nabla^\perp\cdot\MM{u}+f}{h}, \quad \nabla^\perp\cdot\MM{u}
=-\pp{u_1}{y}+\pp{u_2}{x},
\end{equation}
and we have used the split $(\MM{u}\cdot\nabla)\MM{u}=q\MM{u}^\perp +
\nabla(|\MM{u}|^2/2)$. If we apply $\nabla^\perp$ to Equation
\eqref{u eqn nonlinear} and substitute Equation \eqref{eq:q} we 
obtain 
\begin{equation}
\label{q conservation}
(qh)_t + \nabla\cdot(qh\MM{u})=0,
\end{equation}
which is the conservation law for $q$. This takes an important role in
predicting the large scale balanced flow.

Three issues need to be addressed when discretising these equations
with compatible finite element methods. First, if $\MM{u}\in V_1$ and
$h \in V_2$, then $gh + |\MM{u}|^2/2$ has discontinuities and the
gradient is not globally defined. Second, although we can evaluate
$\nabla\cdot\MM{u}$ globally, we cannot evaluate
$\nabla\cdot(h\MM{u})$ globally, as $h$ is discontinuous. Third, we
need a way of calculating $q$. The first issue is addressed by
using integration by parts, as in the linear case. The second issue
can be addressed by projecting $\MM{u}h$ into $V_1$, \emph{i.e.}
solving for $\MM{F}\in V_1$ such that
\begin{equation}
\label{eq:F discrete}
\int_\Omega \MM{w}\cdot\MM{F}\diff x = \int_\Omega \MM{w}\cdot
h\MM{u}\diff x, \quad \forall \MM{w} \in V_1.
\end{equation}
To calculate $q$, we need to integrate the $\nabla^\perp\cdot$
operator on $\MM{u}$ by parts, since $\MM{u}\in V_1$ has insufficient
continuity for $\nabla^\perp\cdot\MM{u}$ to be globally defined. We
choose $q\in V_0$, and multiply Equation \eqref{eq:q} by $h$, then
$\gamma\in V_0$, then finally integrate, to obtain
\begin{equation}
\label{eq:q discrete}
\int_\Omega \gamma q h \diff x = -\int_\Omega \nabla^\perp \cdot \gamma
\MM{u}\diff x + \int_\Omega \gamma f \diff x, \quad
\forall \gamma \in V_0.
\end{equation}
If $h$ is known, then this equation can be solved for $q\in V_0$ (the
factor of $h$ just reweights the integral in each element).

Having addressed these three issues, we can write down the 
compatible finite element discretisation of the nonlinear 
shallow water equations:
\begin{align}
\label{eq:u nonlinear discrete}
\int_\Omega \MM{w}\cdot\MM{u}_t \diff x + \int_\Omega \MM{w}\cdot
q \MM{F}^\perp\diff x + \int_\Omega \nabla\cdot\MM{w}
\left(gh + \frac{1}{2}|\MM{u}|^2\right)\diff x &= 0, \\
\int_\Omega \phi \left(h_t + \nabla\cdot\MM{F}\right)\diff x &=0, 
\label{eq:h nonlinear discrete}
\end{align}
where $\MM{F}$ and $q$ are defined from Equations \eqref{eq:F
  discrete} and \eqref{eq:q discrete} respectively. There are a number
of things to observe about these equations. Firstly, the following quantities
are conserved:
\begin{align}
\mbox{Mass:}\qquad & \int_\Omega h \diff x, \\
\mbox{Energy:}\qquad & \int_\Omega \frac{h}{2}\left(|\MM{u}|^2 + gh\right)\diff x, \\
\mbox{Total vorticity:}\qquad & \int_\Omega qh \diff x, \\
\mbox{Enstrophy:}\qquad & \int_\Omega q^2 h \diff x.
\end{align}
To see that mass is conserved, just take $\phi=1$ in Equation
\eqref{eq:h nonlinear discrete}, and the divergence integrates to 0 by
the Divergence Theorem. To see that the total vorticity is conserved,
take $\gamma=1$ in Equation \eqref{eq:q discrete}. The $\MM{u}$ term
vanishes since $\nabla^\perp\gamma=0$, and $f$ is independent of time.
Similar direct computations lead to conservation of energy and
enstrophy; these make use of the integral formulation and the
compatibility properties and are presented in \cite{McRaCo2014}
together with numerical verifications of the conservation properties.

It is also interesting to ask what equation $q$ satisfies in the
discrete setting. Our prognostic variables are $u$ and $h$, with $q$
being purely diagnostic, so we have to make use of the $u$ and $h$
equations to obtain the dynamical equation for $q$. To do this,
we apply a time derivative to Equation \eqref{eq:q discrete}, and obtain
\begin{equation}
\int_\Omega \gamma (qh)_t \diff x + \int_\Omega \nabla^\perp\gamma
\cdot \MM{u}_t\diff x  = 0, \quad \forall \gamma \in V_0.
\end{equation}
Since $\nabla^\perp\gamma \in V_1$, we can substitute $\MM{w}=\nabla^\perp
\gamma$ in Equation \eqref{eq:u nonlinear discrete}, and we get
\begin{align}
\int_\Omega \nabla^\perp\gamma
\cdot \MM{u}_t\diff x &= -\int_\Omega \nabla^\perp\gamma\cdot q\MM{F}^\perp\diff x
-\int_\Omega \underbrace{\nabla\cdot\nabla^\perp\gamma}_{=0}\left(|\MM{u}|^2 + gh\right)\diff x, \\
&= -\int_\Omega \nabla\gamma\cdot q\MM{F}\diff x,
\end{align}
and hence
\begin{equation}
\int_\Omega \gamma (qh)_t \diff x - \int_\Omega \nabla\gamma
\cdot \MM{F}q\diff x  = 0, \quad \forall \gamma \in V_0.
\end{equation}
Finally, since $\MM{F}$ has continuous normal components and
$\gamma$ is continuous, we may integrate by parts without 
changing the finite element discretisation, and we obtain
\begin{equation}
\int_\Omega \gamma \left((qh)_t + \nabla
\cdot \MM{F}q\right)\diff x  = 0, \quad \forall \gamma \in V_0.
\end{equation}
This is the projection of Equation \eqref{q conservation} into $V_0$,
and so the discretisation has a consistent potential vorticity
conservation law.

It should be noted that for shallow water equations in the geostrophic
limit, it is desirable to dissipate enstrophy at the grid scale rather
than conserve it exactly, due to the enstrophy cascade to small scales
which would otherwise cause gridscale oscillations. In
\cite{McRaCo2014}, \eqref{eq:u nonlinear discrete} was modified to
dissipate enstrophy at the gridscale whilst conserving energy; this
follows the Anticipated Potential Vorticity Method strategy of
\cite{ArHs1990}. It was shown in \cite{McRaCo2014} that this
modification leads to stable vortex merger solutions that do not
develop gridscale oscillations. Further, it is desirable to replace
APVM by stable, accurate upwind advection schemes for $q$ and $h$; we
are developing the integration of discontinuous Galerkin methods for
$h$ and higher-order Taylor-Galerkin methods for $q$ in current work.

Finally, we present some results integrating the shallow water
equations on the sphere using Test Case 5 (the mountain test case) as
specified in \cite{Wi1992}. Figure \ref{fig:w5_conv} gives a
convergence plot upon comparing the height field with the solution
from a resolved pseudospectral calculation, using the BDFM1 space with
a successively refined icosahedral mesh. The expected 2nd order
convergence is obtained. Figure \ref{fig:day_15} is an image of the
velocity and height fields at day 15, while Figure \ref{fig:pv} shows
the evolution of the potential vorticity field out to 50 days.

\begin{figure}
  \centering
  \includegraphics[width=0.8\textwidth]{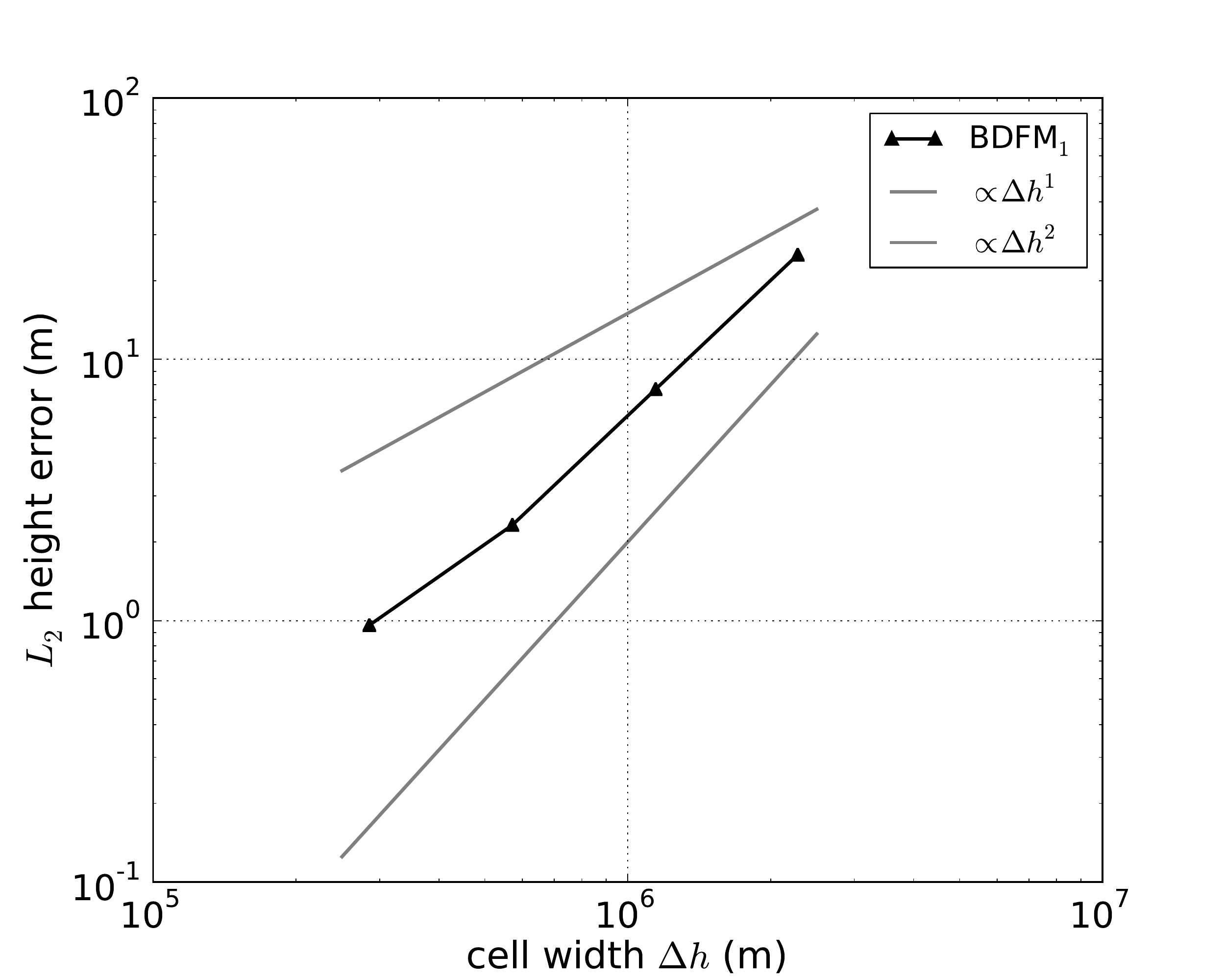}
  \caption{\label{fig:w5_conv} $\|h  - h_\mathrm{ref} \|$
    versus mesh size (where $h_\mathrm{ref}$ is a reference
    solution and $h$ is the numerical solution) for the Williamson 5
    test case. $\Delta t = 225s$, 4 quasi-Newton iterations per time step.}
\end{figure}

\begin{figure}
  \centering
  \includegraphics[width=0.45\textwidth]{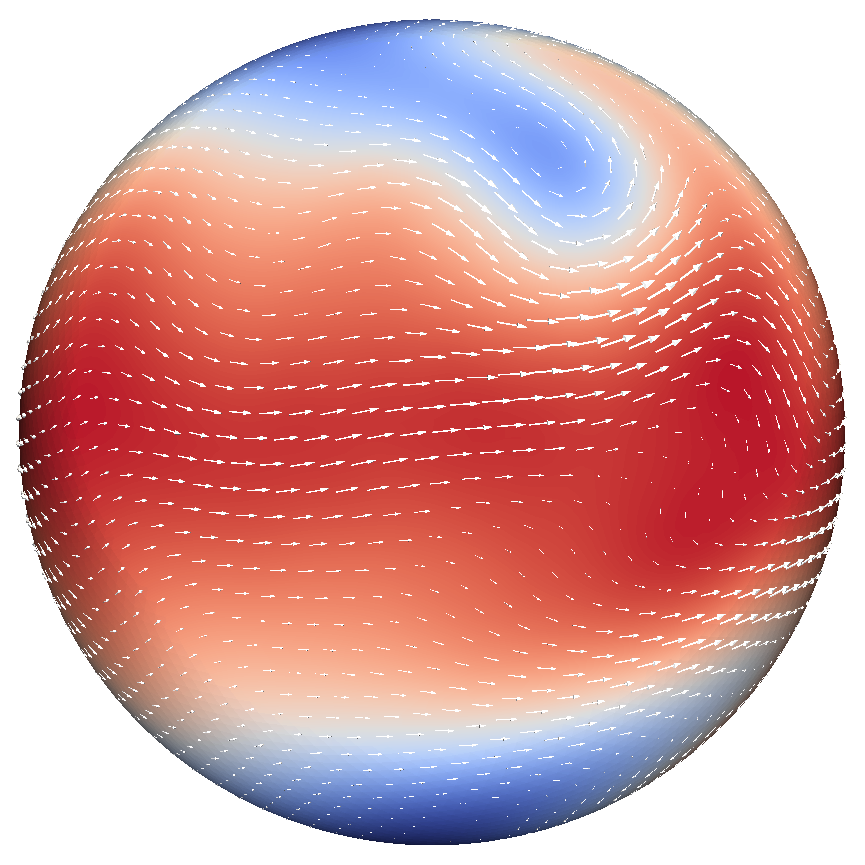}
  \includegraphics[width=0.45\textwidth]{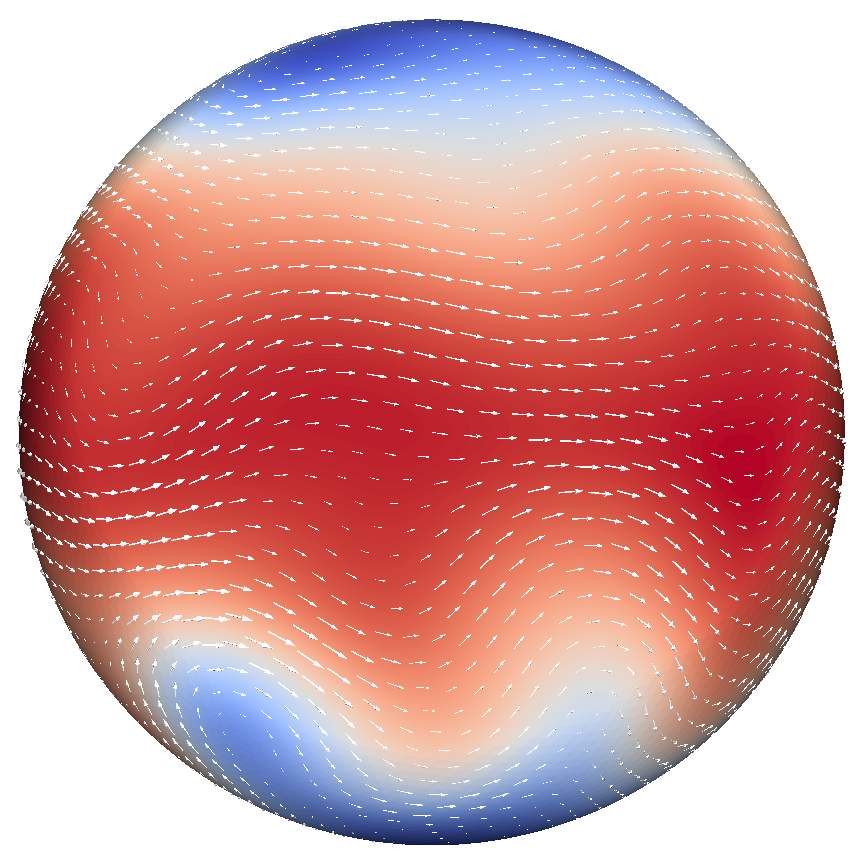}

  \caption{\label{fig:day_15} Snapshots of the velocity and height fields
  in the Williamson 5 test case at 15 days. Blue represents small fluid
  depth, red represents large fluid depth. Left: facing the mountain. Right:
  reverse side.}
\end{figure}

\begin{figure}
  \centering
  \includegraphics[width=0.35\textwidth]{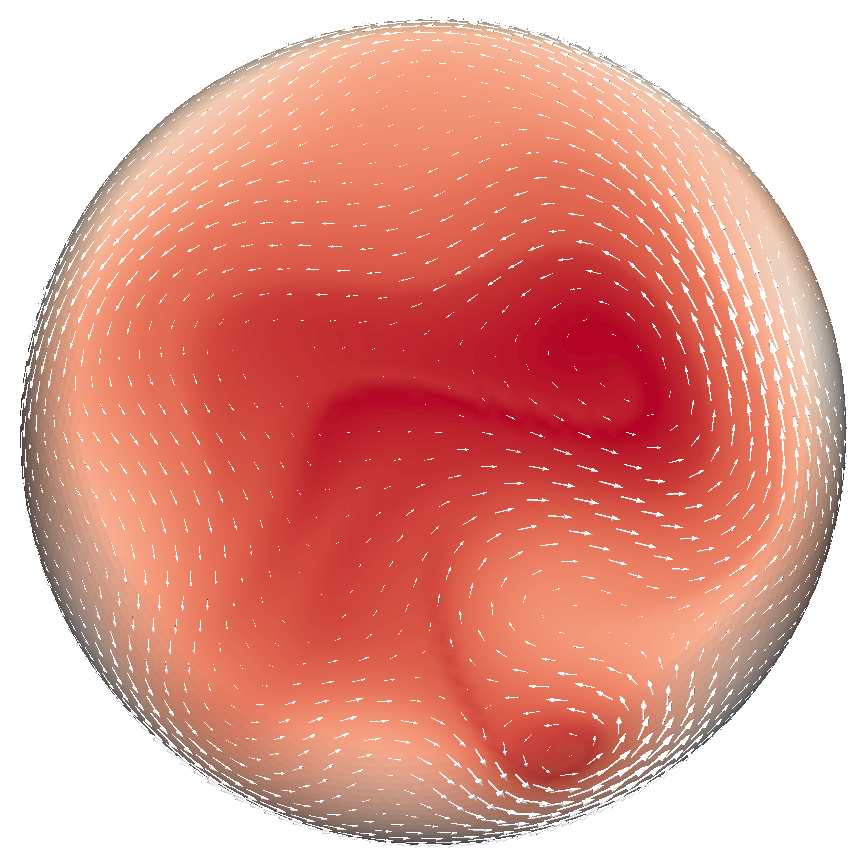}
  \includegraphics[width=0.35\textwidth]{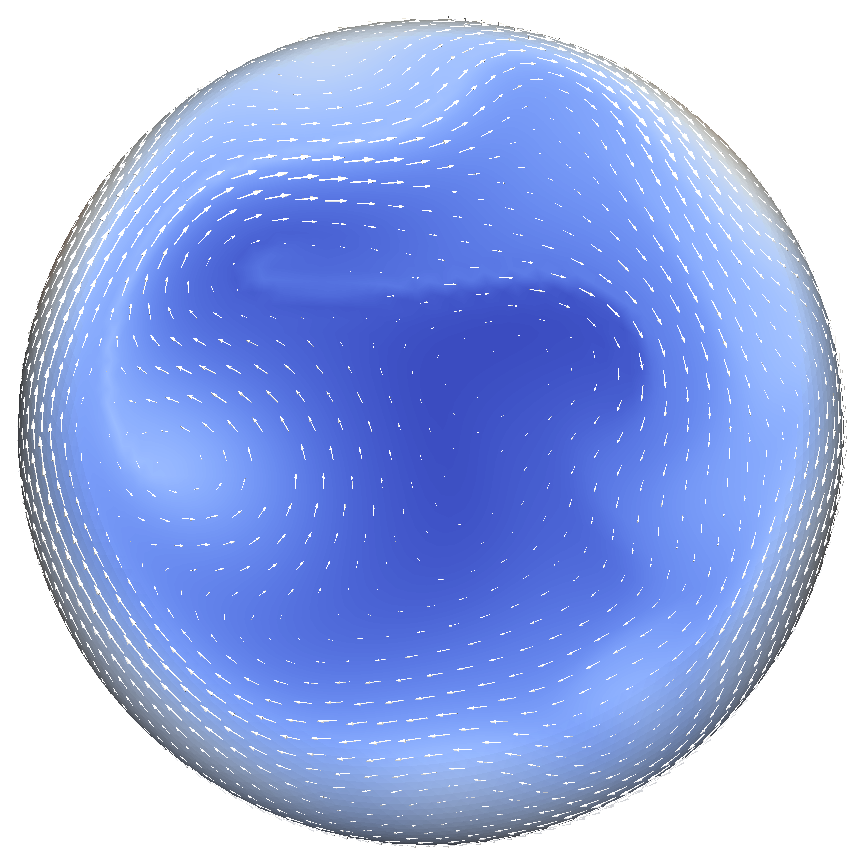}
  \\
  \includegraphics[width=0.35\textwidth]{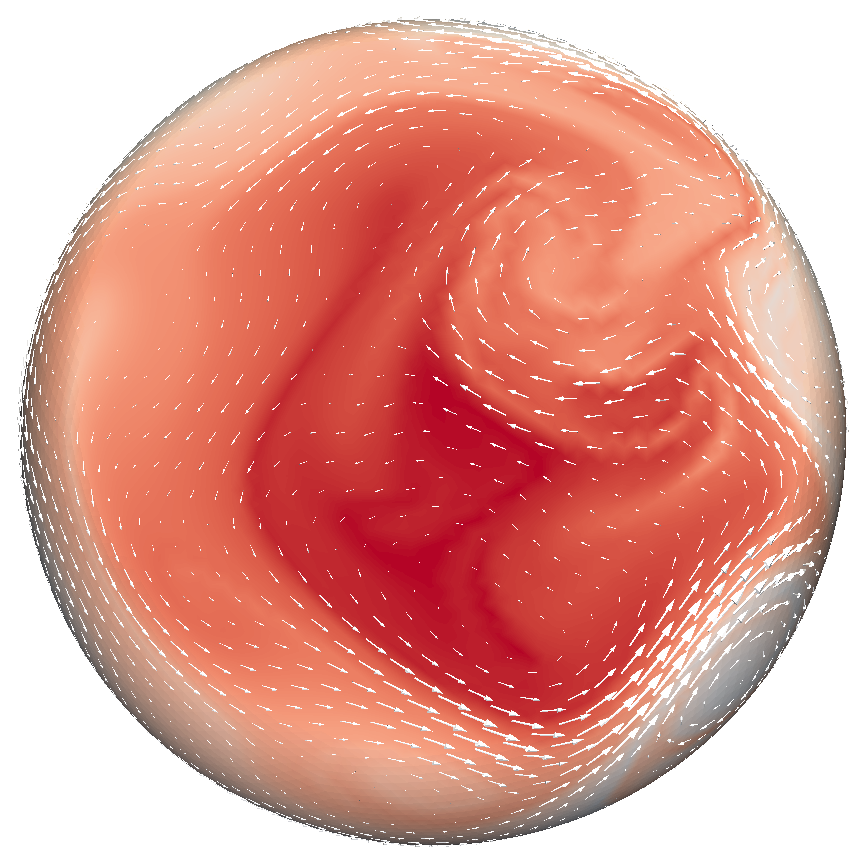}
  \includegraphics[width=0.35\textwidth]{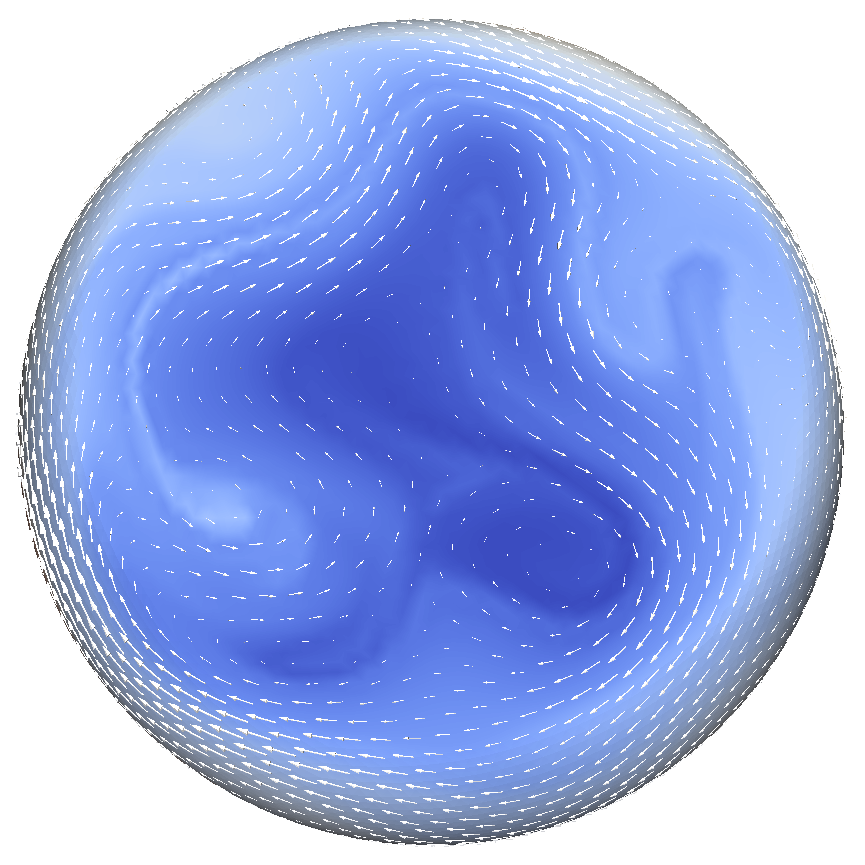}
  \\
  \includegraphics[width=0.35\textwidth]{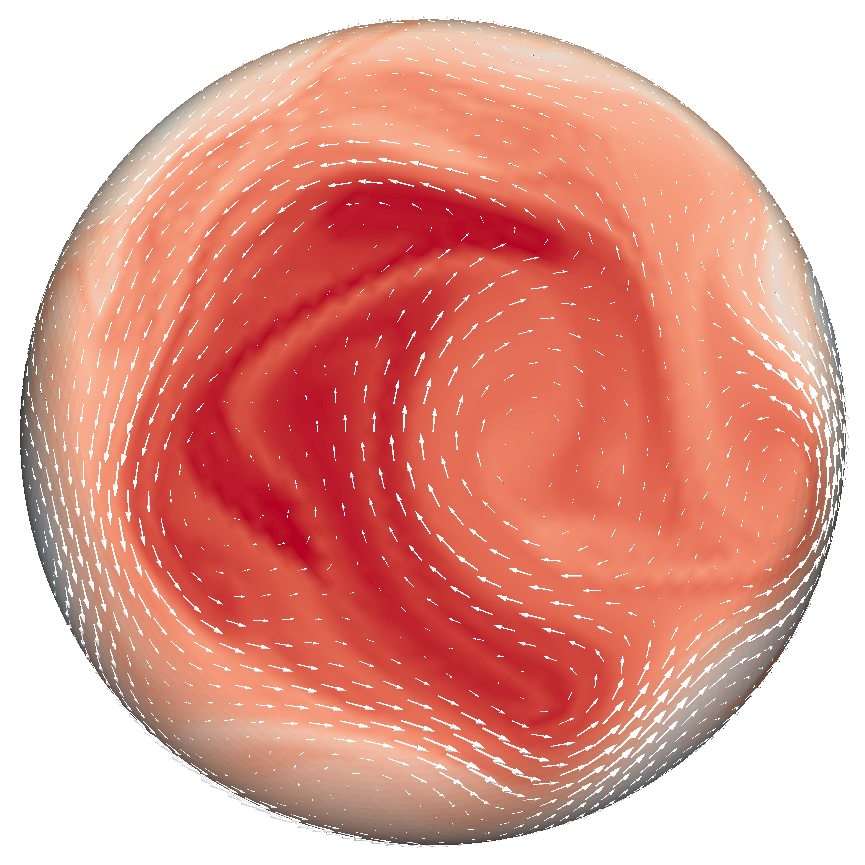}
  \includegraphics[width=0.35\textwidth]{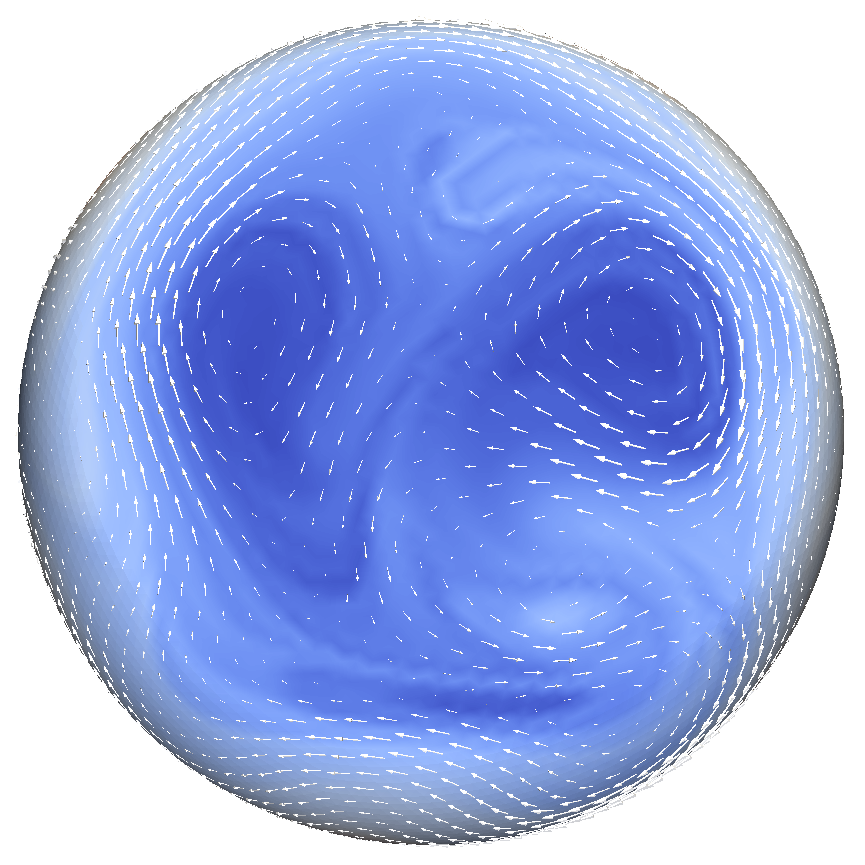}
  \\
  \includegraphics[width=0.35\textwidth]{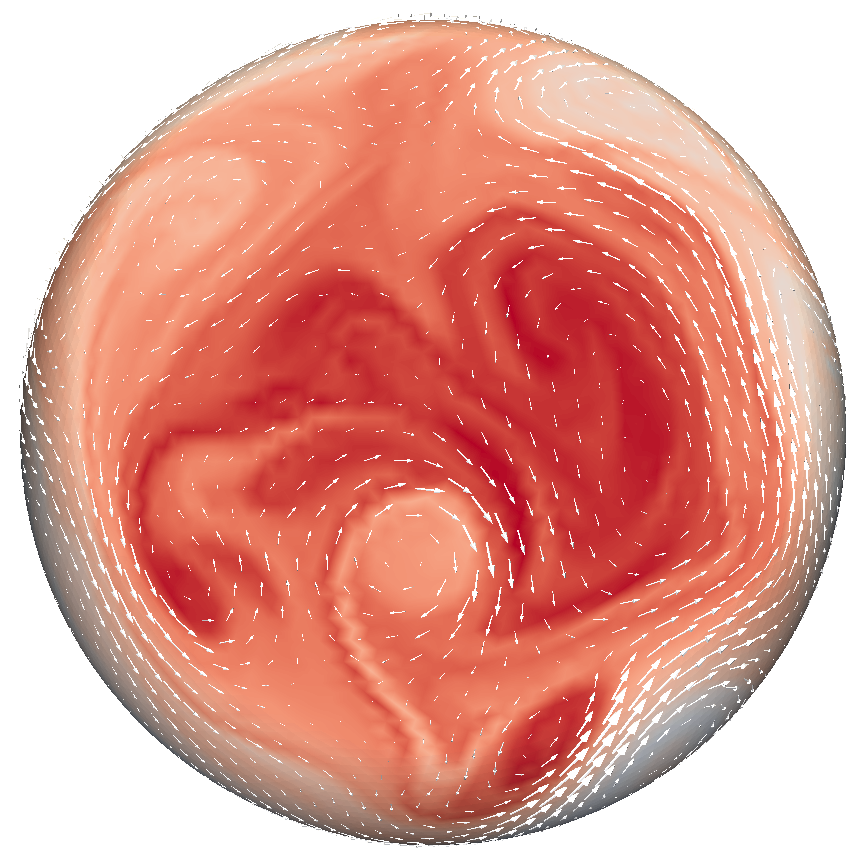}
  \includegraphics[width=0.35\textwidth]{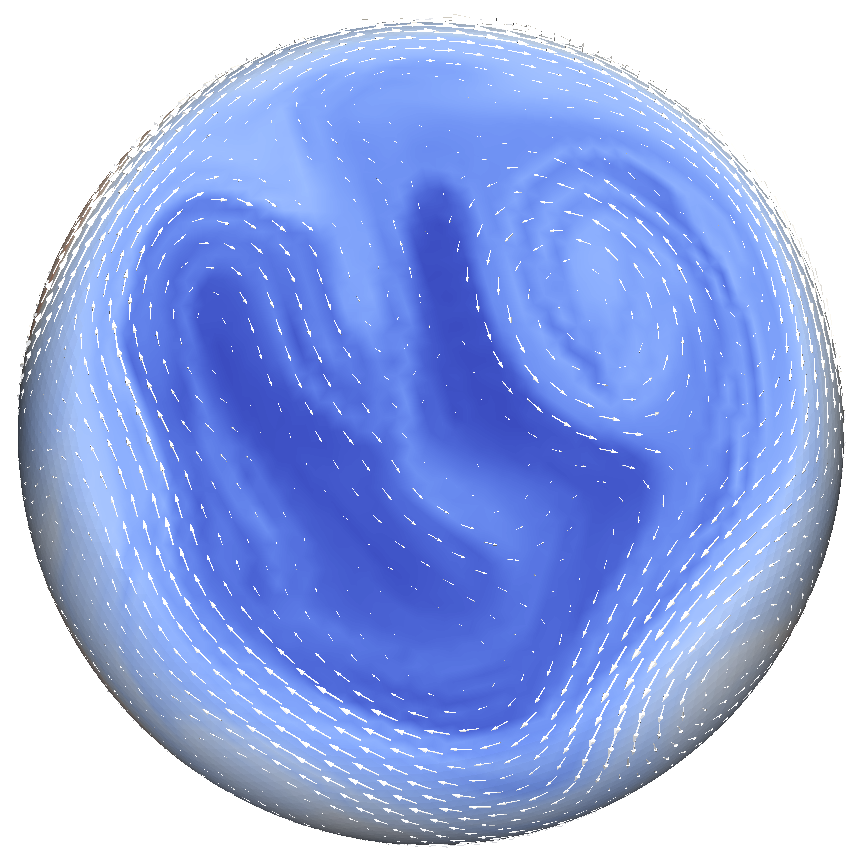}
  \caption{\label{fig:pv} Top to bottom: snapshots of the potential
    vorticity field in the Williamson 5 test case, at 20, 30, 40 and
    50 days respectively, with superposed (scaled) velocity
    vectors. Left: facing North pole. Right: facing South pole.}
\end{figure}

\section{Outlook on applications in three dimensional models}
\label{outlook}
Current work as part of the UK GungHo dynamical core project is
investigating the application of compatible finite element spaces to
three dimensional compressible flow. In three dimensions, we now
have four finite element spaces, and the required structure 
is depicted in the following diagram
  \begin{equation*}
  \begin{CD}
   \underbrace{{V}_0}_{\mbox{\tiny Continuous}}@>\nabla >>
    \underbrace{{V}_1}_{\mbox{\tiny Continuous tangential components}} @>\nabla\times >>
    \underbrace{{V}_2}_{\mbox{\tiny Continuous normal components}} @>\nabla\cdot>>
    \underbrace{{V}_3}_{\mbox{\tiny Discontinuous}}
  \end{CD}
\end{equation*}
where $V_1$ and $V_2$ are both vector-valued finite element spaces.
There is also an extension of the vector invariant form into three
dimensions. We choose $\MM{u}\in V_2$, $\rho \in V_3$, and it is
possible to define a vorticity $\MM{\omega}\in V_1$ using integration
by parts.  For a vertical discretisation similar to the Lorenz grid,
we could choose potential temperature $\theta$ to be in $V_3$, however
the extension of the Charney-Phillips grid would require $\theta$ to
be in the vertical part of $V_2$; this is the subject of current work.
Other challenges in this setting include determining the correct form
of the pressure gradient term, the treatment of the velocity advection
(which would be \emph{via} an implied vorticity equation), and the
efficient solution of the coupled linear system that is required for a
semi-implicit implementation.

\bibliography{cjc_ecmwf}

\begin{thebibliography}{33}
\expandafter\ifx\csname natexlab\endcsname\relax\def\natexlab#1{#1}\fi
\expandafter\ifx\csname url\endcsname\relax
  \def\url#1{\texttt{#1}}\fi
\expandafter\ifx\csname urlprefix\endcsname\relax\def\urlprefix{URL }\fi

\bibitem[{Allen et~al.(1985)Allen, Ewing, and Koebbe}]{allen1985mixed}
Allen, M.~B., Ewing, R.~E., Koebbe, J., 1985. Mixed finite element methods for
  computing groundwater velocities. Numerical Methods for Partial Differential
  Equations 1~(3), 195--207.

\bibitem[{Arakawa and Hsu(1990)}]{ArHs1990}
Arakawa, A., Hsu, Y.-J.~G., 1990. Energy conserving and potential-enstrophy
  dissipating schemes for the shallow water equations. Monthly Weather Review
  118~(10), 1960--1969.

\bibitem[{Arakawa and Lamb(1981)}]{ArLa1981}
Arakawa, A., Lamb, V., 1981. A potential enstrophy and energy conserving scheme
  for the shallow water equations. Monthly Weather Review 109~(1), 18--36.

\bibitem[{Arnold et~al.(2006)Arnold, Falk, and Winther}]{ArFaWi2006}
Arnold, D., Falk, R., Winther, R., 2006. Finite element exterior calculus,
  homological techniques, and applications. Acta Numerica 15, 1--155.

\bibitem[{Arnold et~al.(2010)Arnold, Falk, and Winther}]{ArFaRa2010}
Arnold, D., Falk, R., Winther, R., 2010. Finite element exterior calculus: from
  {Hodge} theory to numerical stability. Bull. Amer. Math. Soc.(NS) 47~(2),
  281--354.

\bibitem[{Auricchio et~al.(2004)Auricchio, Brezzi, and Lovadina}]{AuBrLo2004}
Auricchio, F., Brezzi, F., Lovadina, C., 2004. Mixed Finite Element Methods.
  Vol.~1. Wiley, Ch.~9.

\bibitem[{Boffi et~al.(2013)Boffi, Brezzi, and Fortin}]{boffi2013mixed}
Boffi, D., Brezzi, F., Fortin, M., 2013. Mixed finite element methods and
  applications. Springer.

\bibitem[{Bossavit(1988)}]{bossavit1988whitney}
Bossavit, A., 1988. Whitney forms: a class of finite elements for
  three-dimensional computations in electromagnetism. IEE Proceedings A
  (Physical Science, Measurement and Instrumentation, Management and Education,
  Reviews) 135~(8), 493--500.

\bibitem[{Brezzi and Fortin(1991)}]{BrFo1991}
Brezzi, F., Fortin, M., 1991. Mixed and hybrid finite element methods.
  Springer-Verlag New York, Inc.

\bibitem[{Comblen et~al.(2010)Comblen, Lambrechts, Remacle, and
  Legat}]{comblen2010practical}
Comblen, R., Lambrechts, J., Remacle, J.-F., Legat, V., 2010. Practical
  evaluation of five partly discontinuous finite element pairs for the
  non-conservative shallow water equations. International Journal for Numerical
  Methods in Fluids 63~(6), 701--724.

\bibitem[{Cotter and Ham(2011)}]{CoHa2011}
Cotter, C., Ham, D., 2011. Numerical wave propagation for the triangular
  {P1DG-P2} finite element pair. Journal of Computational Physics 230~(8), 2806
  -- 2820.

\bibitem[{Cotter and Shipton(2012)}]{cotter2012mixed}
Cotter, C., Shipton, J., 2012. Mixed finite elements for numerical weather
  prediction. Journal of Computational Physics 231~(21), 7076--7091.

\bibitem[{Cotter and Thuburn(2014)}]{CoTh2014}
Cotter, C., Thuburn, J., 2014. A finite element exterior calculus framework for
  the rotating shallow-water equations. J. Comp. Phys. 257, 1506--1526.

\bibitem[{Cotter et~al.(2009)Cotter, Ham, and Pain}]{CoHaPa2009}
Cotter, C.~J., Ham, D.~A., Pain, C.~C., 2009. A mixed discontinuous/continuous
  finite element pair for shallow-water ocean modelling. Ocean Modelling 26,
  86--90.

\bibitem[{Danilov(2010)}]{Da2010}
Danilov, S., 2010. On utility of triangular {C}-grid type discretization for
  numerical modeling of large-scale ocean flows. Ocean Dynamics 60~(6),
  1361--1369.

\bibitem[{Hiptmair(2002)}]{hiptmair2002finite}
Hiptmair, R., 2002. Finite elements in computational electromagnetism. Acta
  Numerica 11~(0), 237--339.

\bibitem[{Karniadakis and Sherwin(2005)}]{KaSh2005}
Karniadakis, G. E.~M., Sherwin, S., 2005. Spectral/hp Element Methods for
  Computational Fluid Dynamics. Oxford Science Publications.

\bibitem[{Le~Roux(2012)}]{le2012spurious}
Le~Roux, D.~Y., 2012. Spurious inertial oscillations in shallow-water models.
  Journal of Computational Physics 231~(24), 7959--7987.

\bibitem[{{Le Roux} and Pouliot(2008)}]{LePo2008}
{Le Roux}, D.~Y., Pouliot, B., 2008. Analysis of numerically induced
  oscillations in two-dimensional finite-element shallow-water models {Part
  II}: Free planetary waves. SIAM Journal on Scientific Computing 30~(4),
  1971--1991.

\bibitem[{{Le Roux} et~al.(2007){Le Roux}, Rostand, and Pouliot}]{LeRoPo2007}
{Le Roux}, D.~Y., Rostand, V., Pouliot, B., 2007. Analysis of numerically
  induced oscillations in {2D} finite-element shallow-water models {Part I}:
  Inertia-gravity waves. SIAM Journal on Scientific Computing 29~(1), 331--360.

\bibitem[{{Le Roux} et~al.(2005){Le Roux}, S\`ene, Rostand, and
  Hanert}]{Ro+2005}
{Le Roux}, D.~Y., S\`ene, A., Rostand, V., Hanert, E., 2005. On some spurious
  mode issues in shallow-water models using a linear algebra approach. Ocean
  Modelling, 83--94.

\bibitem[{McRae and Cotter(to appear)}]{McRaCo2014}
McRae, A., Cotter, C., to appear. Energy-and enstrophy-conserving schemes for
  the shallow-water equations, based on mimetic finite elements.
  QJRMS\,Preprint at \url{http://arxiv.org/abs/1305.4477}.

\bibitem[{Ringler et~al.(2010)Ringler, Thuburn, Klemp, and
  Skamarock}]{RiThKlSk2010}
Ringler, T.~D., Thuburn, J., Klemp, J.~B., Skamarock, W.~C., 2010. {A unified
  approach to energy conservation and potential vorticity dynamics for
  arbitrarily-structured {C}-grids}. Journal of Computational Physics 229~(9),
  3065--3090.
\newline\urlprefix\url{http://dx.doi.org/10.1016/j.jcp.2009.12.007}

\bibitem[{Rognes et~al.(2014)Rognes, Ham, Cotter, and McRae}]{Ro+2013}
Rognes, M., Ham, D., Cotter, C., McRae, A., 2014. Automating the solution of
  {PDE}s on the sphere and other manifolds, to appear in Geosci. Model Dev.

\bibitem[{Rognes et~al.(2009)Rognes, Kirby, and Logg}]{RoKiLo2009}
Rognes, M., Kirby, R., Logg, A., 2009. Efficient assembly of {$H(\mathrm{div})$
  and $H(\mathrm{curl})$} conforming finite elements. SISC 31~(6), 4130--4151.

\bibitem[{Rostand and Le~Roux(2008)}]{rostand2008raviart}
Rostand, V., Le~Roux, D., 2008. {Raviart--Thomas and Brezzi--Douglas--Marini}
  finite-element approximations of the shallow-water equations. International
  journal for numerical methods in fluids 57~(8), 951--976.

\bibitem[{Staniforth and Thuburn(2012)}]{StTh2012}
Staniforth, A., Thuburn, J., 2012. Horizontal grids for global weather and
  climate prediction models: a review. Q. J. Roy. Met. Soc 138~(662A), 1--26.

\bibitem[{Thuburn(2008)}]{Th08}
Thuburn, J., 2008. Numerical wave propagation on the hexagonal {C}-grid. J.
  Comp. Phys. 227~(11), 5836--5858.

\bibitem[{Thuburn and Cotter(2012)}]{ThCo2012}
Thuburn, J., Cotter, C., 2012. A framework for mimetic discretization of the
  rotating shallow-water equations on arbitrary polygonal grids. SIAM J. Sci.
  Comp.

\bibitem[{Thuburn et~al.(2009)Thuburn, Ringler, Skamarock, and
  Klemp}]{ThRiSkKl2009}
Thuburn, J., Ringler, T.~D., Skamarock, W.~C., Klemp, J.~B., 2009. Numerical
  representation of geostrophic modes on arbitrarily structured {C}-grids. J.
  Comput. Phys. 228, 8321--8335.

\bibitem[{Walters and Casulli(1998)}]{WaCa1998}
Walters, R., Casulli, V., 1998. A robust, finite element model for hydrostatic
  surface water flows. Communications in Numerical Methods in Engineering 14,
  931--940.

\bibitem[{Wathen(1987)}]{wathen1987realistic}
Wathen, A., 1987. Realistic eigenvalue bounds for the {Galerkin} mass matrix.
  IMA Journal of Numerical Analysis 7~(4), 449--457.

\bibitem[{Williamson et~al.(1992)Williamson, Drake, Hack, Jakob, and
  Swarztrauber}]{Wi1992}
Williamson, D.~L., Drake, J.~B., Hack, J.~J., Jakob, R., Swarztrauber, P.~N.,
  1992. A standard test set for numerical approximations to the shallow water
  equations in spherical geometry. Journal of Computational Physics 102~(1),
  211--224.

\end{thebibliography}
\end{document}